\documentclass[10pt,cd]{article}
\usepackage[left=1.5in,right=1.5in,top=1in,bottom=1in]{geometry}
\usepackage{graphicx}

\usepackage{amsmath}
\usepackage{amsthm}
\usepackage{amssymb}
\usepackage{amsfonts}
\usepackage{fancyheadings}

\usepackage[all]{xy}
\xyoption{poly} \xyoption{graph} \xyoption{arc} \xyoption{web}
\xyoption{2cell}
\usepackage{float}

\makeatletter
\newtheorem*{rep@theorem}{\rep@title}
\newcommand{\newreptheorem}[2]{%
\newenvironment{rep#1}[1]{%
\def\rep@title{#2 \ref{##1}}%
\begin{rep@theorem}}%
{\end{rep@theorem}}}
\makeatother

\newtheorem{theorem}{Theorem}[section]
\newreptheorem{theorem}{Theorem}
\newtheorem{lemma}[theorem]{Lemma}

\newtheorem{corollary}[theorem]{Corollary}

\newtheorem{example}[theorem]{Example}

\DeclareMathOperator{\supp}{supp}
\DeclareMathOperator{\Aut}{Aut}
\DeclareMathOperator{\A}{Aut}
\DeclareMathOperator{\madf}{m.a.d.f.}
\DeclareMathOperator{\ad}{a.d.}

\newcommand{\Z}{{\mathbb Z}}
\newcommand{\N}{{\mathbb N}}
\newcommand{\R}{{\mathbb R}}
\newcommand{\Q}{{\mathbb Q}}

\newcommand{\B}{{\mathcal B}}

\newcommand{\C}{{\mathcal C}}

\newcommand{\PP}{{\mathcal P}}

\title{Pathological and Highly Transitive Representations of Free Groups}
\author{J. Bruno\\ National University of Ireland, Galway\\ j.bruno1@nuigalway.ie}

\begin{document}

\maketitle
\begin{abstract}

A well-known result by P. Cameron provides us with a construction of the free group of rank $2^{\aleph_0}$ within the automorphism group of the rationals. We show that the full versatility of doubly transitive automorphism groups is not necessary by extending Cameron's construction to a larger class of permutation groups and generalize his result by constructing \emph{pathological} (permutations of unbounded support) and \emph{$\omega$-transitive} (highly transitive) representations of free groups. In particular, and working solely within ZFC, we show that any \emph{large} subgroup of $\Aut(\Q)$ (resp. $\Aut(\R)$) contains an $\omega$-transitive and pathological representation of any free group of rank $\lambda \in [\aleph_0 ,2^{\aleph_0}]$ (resp. of rank $2^{\aleph_0}$). Assuming the continuum to be a regular cardinal, we show that pathological and $\omega$-transitive representations of uncountable free groups abound within large permutation groups of linear orders. Lastly, we also find a bound on the rank of free subgroups of certain restricted direct products.

\end{abstract}
 \section*{Introduction}

 The study of automorphism groups of linear orders effectively began in 1957 with a publication by P. M. Cohn \cite{MR0091280} where he solved a question posed by B. H. Neumann: \textit{Can the automorphism group, $\A(\Omega)$, of any linear order, $\Omega$, be ordered?} An \textit{ordering} on $\A(\Omega)$ (\cite{MR2796249}, \cite{MR0158009}) is meant to convey a partial order $(\A(\Omega), <)$ so that the order is invariant under actions of $\A(\Omega)$ onto itself. Cohn answered this question negatively and gave sufficient and necessary conditions on $\Omega$ for rendering an order on $\A(\Omega)$ and any \textit{large} (a doubly transitive $l$-group closed under piecewise patching and disjoint patching, \cite{MR988099}) subgroups thereof. 
 Under the assumption of the Generalized Continuum Hypothesis (GCH), Glass (\cite{MR0409309}) showed that the free $l$-group $F_{\aleph_\alpha}$ of rank $\aleph_\alpha$ can be represented as an $o$-2 transitive $l$-subgroup of the automorphism group of an $\alpha$-{set}. Over a decade later McCleary (\cite{MR787955}) crystalized further group theoretic properties of free $l$-groups by extending the above result to any free $l$-group. McCleary's remarkable result was achieved by exploiting, as he calls it,  \emph{the best of both worlds}; a right ordering $(G_\kappa, \leq)$ (i.e. a right ordering on the free group of rank $\kappa$) for which the natural action of  $F_\kappa$ (in the sense of Conrad \cite{MR0270992}) is faithful and $o$-2 transitive. A comprehensive retrospective survey by Bludov, Droste and Glass can be found in \cite{MR2796250}. Working in the other direction (i.e. when do permutation groups of linear orders accept free subgroups), Cameron (\cite{Cameron99oligomorphicpermutation}) constructs a copy of the free group of rank $2^{\aleph_0}$ within $\Aut(\mathbb{Q})$. Moreover, it is possible to show that any doubly transitive automorphism group of a linear order (one which acts transitively on ordered pairs) of a linear order must contain a copy of $\Aut(\mathbb{Q})$.
 In this paper we show that the full versatility of doubly transitive automorphism groups is not necessary by constructing \emph{pathological} (permutations of unbounded support) and \emph{$\omega$-transitive} (transitive on all ordered $n$-tuples) representations of free groups within several classes of doubly transitive permutation groups of linear orders. In particular, and as a consequence of our main result, we greatly generalize Cameron's construction by proving that any \emph{large} subgroup of $\Aut(\Q)$ (resp. $\Aut(\R)$) contains an $\omega$-transitive and pathological representation of any free group of rank $\lambda \in [\aleph_0 ,2^{\aleph_0}]$ (resp. of rank $2^{\aleph_0}$).

 \section*{Background}
 The symbols $\kappa$ and $\lambda$ will always denote infinite cardinals and $\mathfrak{c} = 2^{\aleph_0}$. For any $\kappa$, $\lambda = \kappa^{+n}$ denotes the n$^{th}$ successor of $\kappa$ and  $[\kappa]^{<\omega}$ denotes the collection of all of its finite subsets. A function $f: \kappa \rightarrow \lambda$ is \emph{cofinal} in $\lambda$ provided its range ($ran(f)$) knows no bounds in $\lambda$. The \emph{cofinality} of any $\lambda$ (cof($\lambda$)) is the smallest $\kappa$ so that there exists a function that cofinally maps $\kappa$ into $\lambda$. A cardinal is \emph{regular} if it is its own cofinality and singular otherwise. Since it is consistent with ZFC that $\mathfrak{c} = \aleph_1$ and that $\mathfrak{c} = \aleph_{\omega_1}$, regularity of $\mathfrak{c}$ is then independent of ZFC. If $x,y \subset \kappa$ so that $|x| = |y| = \kappa$ and $|x \cap y| < \kappa$ then $x$ and $y$ are said to be \emph{almost disjoint} ($\ad$). An $\ad$ \emph{family} is a collection of pairwise $\ad$ sets and any such family is \emph{maximal} ($\madf$) whenever it is not contained in any other $\ad$ family.

Standard definitions and notation regarding permutation groups (i.e. \emph{$l$-group, faithful representation, support}, etc) can be found in \cite{MR2796250}. In order to avoid confusion, we remark that throughout this paper all $k$-tuples are assumed to be ordered and the symbols $P$ and $\Omega = (\Omega, \leq)$ will represent a permutation group and a linear order, respectively. A \textit{k-transitive} permutation group $P$ on an arbitrary $\Omega$ is one for which given any pair of $k$-tuples $(a_1, a_2, \ldots, a_k),(b_1, b_2, \ldots, b_k)\in \Omega^k$ we can find an $f \in P$ for which $(f(a_1), f(a_2), \ldots, f(a_k))= (b_1, b_2, \ldots, b_k)$. Moreover, if $P$ is $k$-transitive for all $k \in \N$ then we refer to it as an \emph{$\omega$-transitive} permutation group. A $2$-transitive permutation group will be called \emph{doubly transitive}. Notice that if $P$ is doubly transitive and $\Omega$ contains more than two points then $\Omega$ must be a dense linear order (DLO) and have no end points. A faithful representation (in the sense of \cite{MR787955}) of a group $G$ within a permutation group will be denoted by $\hat{G}$. We adopt the traditional use of $G_\kappa$ to denote the free group of rank $\kappa$. For a $g \in P$, $\supp(g)$ will denote its support and $P$ will be referred to as \emph{pathological} provided it contains no element ($\not = e$) of bounded support.  

Recall that any $P \leq \Aut(\Omega)$ is said to be closed under \emph{piecewise patching} precisely when given any convex subset $S$ of $\Omega$ and coterminal (unbounded above and below) sequences $\{a_i \mid i \in \Z\}$ and $\{b_i \mid i \in \Z\}$ in $S$ with $a_i < a_{i+1}$ and $b_i < b_{i+1}$ so that $\forall i \in \Z$ there exists $g_i \in P $ for which $g_i([a_i, a_{i+1}]) = [b_i, b_{i+1}]$ then $P$ also contains an element $g$ that acts as the identity outside $S$ and $g(x) = g_i(x)$ provided $x \in [a_i, a_{i+1}]$. Similarly, P is said to be \emph{closed under disjoint patching} if for all $i \in I$, $g_i \in P$ and $\supp(g_i) \cap \supp(g_j) = \emptyset$ (for $j \not = i$) then $\exists g \in P$ so that

\begin{equation*}
g(x) =
\begin{cases}

g_i(x) & \text{if } x \in \supp(g_i)\mbox{ and}\\
x & \text{if } \mbox{ otherwise.}\\
\end{cases}
\end{equation*}\\

Any doubly transitive $l$-subgroup of $\Aut(\Omega)$ closed under piecewise patching and disjoint patching is said to be \emph{large in $\Aut(\Omega)$} (or just \emph{large}). The letter $H$ will be used to refer to large permutation groups exclusively. We adopt the traditional use of the restriction symbol `$\upharpoonright$'. That is, for $\Lambda \subseteq \Omega$ and any $g \in \Aut(\Omega)$ then $g \upharpoonright \Lambda$ (if it exists) is the unique element in $\Aut(\Lambda)$ that acts like $g$ on $\Lambda$. Lastly, for any $P$ and $\Lambda \subseteq \Omega$ we define $\Aut_P(\Lambda) := \{ g\in \Aut(\Lambda) \mid \exists$  $\overline{g} \in P$ so that $\overline{g} \upharpoonright \Lambda = g\}$ Notice that $\Aut_P(\Lambda)$ makes no reference to the ambient linear order $\Omega$, hence this symbol will only be used when it is clear what $\Omega$ is, leaving no room for confusion.

Our method for constructing representations of free groups is supported by two keystones: one is Cameron's representation of $G_\mathfrak{c}$ within $\Aut(\R)$ while the other is a consequence of the following lemma (\cite{pingpong}).

 \begin{lemma}[Ping-Pong Lemma] Let $G$ be a group acting on a set $X$. Suppose that $\{A_1, A_2, \ldots, A_n, B_1, B_2, \ldots, B_n\}$ is a set of pairwise disjoint subsets of $X$ so that for \\$f_1, f_2, \ldots, f_n \in G$ we have:

 $$ B_i^c \subseteq f_i(A_i),$$

then the group generated by $\{f_1, f_2, \ldots, f_n\}$ is free.

 \end{lemma}

\section{The Countable Pathological Case}

We begin by illustrating a simple representation of  $G_2$ within $\Aut(\R$) for it can be easily extended to any large permutation group of a linear order\footnote{There are several constructions of free groups within $\Aut(\R)$. Some are explicit constructions \cite{MR1363412} while others involve more difficult algebraic notions \cite{MR969681}.}.

 \begin{example} \label{exa:example} Let $N_n = \{[a, a+1) \mid a \equiv n$ $(mod$ $4)\}$ ($n = 0,1,2,3$) and define $f \in \Aut(\R)$ so that $f[a,a+1) = [a, a+\frac{13}{4}), [a+\frac{9}{4}, a+\frac{10}{4}), [a + \frac{6}{4}, a + \frac{7}{4})$ and $[a+ \frac{3}{4}, a + 1)$ when $[a, a+1) \in N_0, N_1, N_2$ and $N_3$ respectively. Since $f$ is periodic with a period of $4$ we can understand how $f$ behaves on $\R$ by its action on the interval $[0,4]$.
  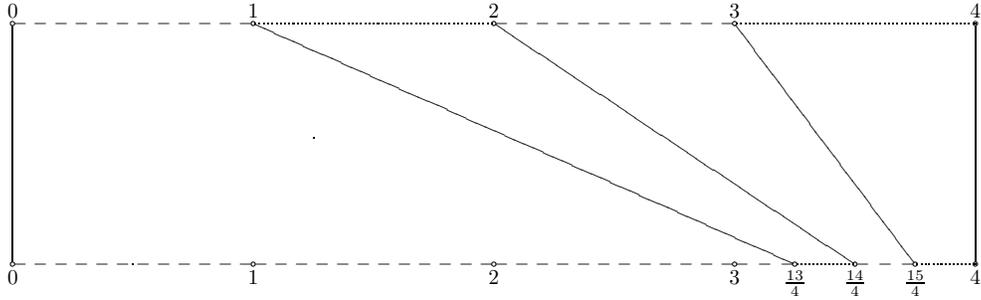
\begin{figure}[H]
\begin{center}
\scalebox{0.8}{$\begin{xy}
 \POS
(-40,10) *\cir<1pt>{} ="0" *+!D{0},
(0,10) *\cir<1pt>{} ="1" *+!D{1},
(40,10) *\cir<1pt>{} ="2" *+!D{2},
(80,10) *\cir<1pt>{} ="3" *+!D{3},
(120,10) *\cir<1pt>{} ="4" *+!D{4},
(120,10) *\cir<0pt>{} ="0R",
(-20,-30) *\cir<0pt>{} ="0DL",
(10,-9) *\cir<0pt>{} ="a'",
(-40,-30) *\cir<1pt>{} ="a" *+!U{0},
(0,-30) *\cir<1pt>{} ="b" *+!U{1},
(40,-30) *\cir<1pt>{} ="c" *+!U{2},
(80,-30) *\cir<1pt>{} ="d" *+!U{3},
(110,-30) *\cir<1pt>{} ="15/4" *+!U{\frac{15}{4}},
(100,-30) *\cir<1pt>{} ="14/4" *+!U{\frac{14}{4}},
(90,-30) *\cir<1pt>{} ="13/4" *+!U{\frac{13}{4}},
(120,-30) *\cir<1pt>{} ="e" *+!U{4},
(120,-30) *\cir<0pt>{} ="0DR",

\POS "0" \ar@{--}  "1",
\POS "1" \ar@{..}  "2",
\POS "2" \ar@{--}  "3",
\POS "3" \ar@{..}  "4",
\POS "4" \ar@{..}  "0R",

\POS "a" \ar@{--}  "b",
\POS "b" \ar@{--}  "c",
\POS "c" \ar@{--}  "d",
\POS "13/4" \ar@{..}  "14/4",
\POS "14/4" \ar@{--}  "15/4",
\POS "15/4" \ar@{..}  "e",
\POS "d" \ar@{--}  "13/4",
\POS "e" \ar@{..}  "0DR",

\POS "0" \ar@{-}  "a",
\POS "2" \ar@{-}  "14/4",
\POS "3" \ar@{-}  "15/4",
\POS "1" \ar@{-}  "13/4",
\POS "4" \ar@{-}  "e",

 \end{xy}$}
\caption{The interval $[0,4]$ under $f$.}
\end{center}
\end{figure}

 As for $g \in \Aut(\R)$ we let $g(x) = f(x-2)+2$. It is then possible to apply the Ping-Pong Lemma to this example by letting $A_1 = N_0$, $B_1 = N_3$, $A_2 = N_2$ and $B_2 = N_1$. In turn, $\{f,g\}$ generates a free subgroup of $\Aut(\R)$.
 \end{example}

Notice that the above example makes no use of the algebraic properties of $\R$; this is essential since our goal is to extend the above example to linear orders in general. In the sequel we use $G_\omega$ to denote $G_{\aleph_0}$ and use the symbol $\Omega$ exclusively to denote a linear order.

\begin{lemma} \label{lem:interval} For any doubly transitive $P \leq \Aut(\Omega)$ closed under piecewise patching and any interval $\Lambda \subset \Omega$ there exists a representation $\hat{G}_\eta$ of $G_\eta$ ($1 < \eta \leq \aleph_0$) in $\Aut_P(\Lambda)$ so that any element in $\hat{G}_\eta$ can be trivially extended to an element in $P$. That is, given any $g \in \hat{G}_\eta$ we can find a $\overline{g} \in P$ so that $\overline{g} \upharpoonright \Lambda = g$ and the identity otherwise.
\end{lemma}

\begin{proof} Notice that it suffices to prove the above for $\eta = 2$. Let $P$ and $\Lambda$ satisfy the hypothesis, $\Gamma = \{a_i \in \Lambda \mid i \in \Z$ and $a_i < a_{i+1}\}$ (an order-isomorphic copy of $\Z$ in $\Lambda$) and $C_{\Gamma}$ denote the convex hull of $\Gamma$. In the same spirit as with Example~\ref{exa:example} we let $I_n = \{[a_i, a_{i+1}) \mid a_i \in \Gamma$ and $i \equiv n$  $(mod$ $4)\}$. Since $P$ is doubly transitive we can find, for all $i \equiv 0$ $(mod$ $4)$, $f_i, f_i' \in P$ for which $[a_i, a_{i+1}] \mapsto [a_i, a_{i+3}]$ and $[a_{i-3}, a_i] \mapsto [a_{i-1}, a_i]$ respectively. In turn, if we consider the sets $\{a_i \in \Gamma \mid i \equiv 0$ or $1$ $(mod$ $4)\}$ and $\{a_i \in \Gamma \mid i \equiv 0$ or $3$ $(mod$ $4)\}$ and since $P$ is closed under piecewise patching and $\Gamma$ is coterminal in $C_\Gamma$ then we can find an $\overline{f} \in P$ so that:

\begin{equation*}
\overline{f}(x) =
\begin{cases}

x & \text{if } x \not \in C_{\Gamma},\\
f_i(x) & \text{if } x \in [a_i, a_{i+1}],\mbox{ and}\\
f_i'(x) & \text{if } x \in [a_{i-3}, a_{i}].
\end{cases}
\end{equation*}\\

In much the same way as with $\overline{f}$, let $g_i,g_i' \in P$ for which $[a_i, a_{i+1}] \mapsto [a_i, a_{i+3}]$ and $[a_{i-3}, a_i] \mapsto [a_{i -1}, a_i]$ respectively, provided $i \equiv 2$ $(mod$ $4)$. Considering the sets $\{a_i \in \Gamma \mid i \equiv 2$ or $3$ $(mod$ $4)\}$ and $\{a_i \in \Gamma \mid i \equiv 0$ or $3$ $(mod$ $4)\}$ and by double transitivity of $P$ we have a $\overline{g} \in P$ for which:

\begin{equation*}
\overline{g}(x) =
\begin{cases}

x & \text{if } x \not \in C_{\Gamma},\\
g_i(x) & \text{if } x \in [a_i, a_{i+1}],\mbox{ and}\\
g_i'(x) & \text{if } x \in [a_{i-3}, a_{i}].\\
\end{cases}
\end{equation*}\\

By design, the restrictions of $\overline{g}$ and $\overline{f}$ to $C_\Gamma$ exist; denote them by $g$ and $f$, respectively. By applying the Ping-Pong Lemma to $f$ and $g$ we see that they generate a representation of $G_2$ within $\Aut(C_\Gamma)$; let $A_1 = I_0$, $B_1 = I_3$, $A_2 = I_2$ and $B_2 = I_1$. In turn, their trivial extensions to $\Lambda$ (i.e. $\overline{g} \upharpoonright \Lambda$ and  $\overline{f} \upharpoonright \Lambda$) also generate a representation of $G_2$ within $\Aut(\Lambda)$. To finish the proof we must only notice that $\overline{g}$ and $\overline{f}$ act as the identity outside $\Lambda$.

\end{proof}
In other words, double transitivity in addition to being closed under piecewise patching suffices to construct free groups of countable rank. Moreover, a doubly transitive permutation group implies a highly homogeneous linear order and in view of the previous lemma, free subgroups within such automorphism groups abound.

\begin{corollary} \label{cor:countable} Any doubly transitive subgroup of any $\Aut(\Omega)$  closed under piecewise patching contains a representation of $G_\omega$.
\end{corollary}

The permutation group constructed in Lemma \ref{lem:interval} is not pathological unless the cofinality and coinitiality of $\Omega$ are countable. Of course, it is pathological in the convex hull of $\Gamma$, but we can do better than that. As the next result suggests, being closed under disjoint union is certainly enough.

\begin{lemma} \label{cor:countable} Any doubly transitive $P\leq \Aut(\Omega)$ closed under disjoint patching and piecewise patching contains a pathological representation of $G_\omega$.
\end{lemma}

\begin{proof} Let $P \leq \Aut(\Omega)$ as above and $\C = \{\Lambda_i \subset \Omega \mid i \in I\}$ be any unbounded (below and above) collection of pairwise disjoint intervals. By Lemma~\ref{lem:interval} we know that for each $i \in I$ there exists a representation of $G_\omega$ in $\Aut_P(\Lambda_i)$, say $G_i$, so that any element in $G_i$ can be trivially extended to all of $\Omega$; for each $i \in I$ let $\{g_{i,j}\}_{j \in \omega}$ and $\{\overline{g}_{i,j}\}_{j \in \omega}$ denote the generating set of $G_i$ and their trivial extensions to $\Omega$, respectively. Since $P$ is closed under disjoint patching then for each $j \in \omega$ we can construct an $f_j \in P$ for which $f_j(x) = \overline{g}_{i,j}(x)$ provided that $x \in \Lambda_i$ (for some $i\in I$) and the identity otherwise. This is indeed possible since the supports of any pair $\overline{g}_{k,j},\overline{g}_{l,j}$ (with $k \not = l$) are disjoint. By assumption, $\C$ is unbounded and consequently so is the support of each $f_j$. To this end we must only notice that $\{f_j\}_{j \in \omega}$ generates a pathological representation of $G_\omega$ in $P$.

\end{proof}

We also prove in the sequel that Lemma~\ref{cor:countable} can be further extended up to continuum-size free groups. Moreover, and depending on the transitivity of the permutation group, the representation can be made at least 2-transitive.
\newpage

\section{The Uncountable Case and Transitivity}
As this section illustrates, Lemma~\ref{cor:countable} can be extended much further. That said, the countable case is as far as the Ping-Pong Lemma can take us. For the uncountable case we turn to a construction by Cameron \cite{Cameron99oligomorphicpermutation} of the free group of size continuum within $\Aut(\Q)$. We exploit such a construction by generalizing it to a certain collection of regular cardinals. The following theorem can be found in \cite{MR756630} and is key to what follows.

\begin{theorem} \label{thm:kunen} For $\kappa \geq \omega$ a regular cardinal

\begin{itemize}
\item If $\mathcal{A} \subset \PP(\kappa)$ is an $\ad$ family and $|\mathcal{A}| = \kappa$ then $\mathcal{A}$ is not maximal.
\item There is a $\madf$ within $\PP(\kappa)$ with cardinality $\geq \kappa^+$.
\end{itemize}
\end{theorem}

In the case of $\aleph_0$ it is possible to construct an $\ad$ family within $\PP(\aleph_0)$ of size $2^{\aleph_0}$. It is this fact (along with some insightful ideas) that enables Cameron\footnote{In fact, Cameron credits such a construction to Jim Kister} to construct $G_\mathfrak{c}$ within $\Aut(\Q,<)$ . Upon careful inspection of such a construction it becomes clear that the same can be done with any large subgroup of $\Aut(\Omega)$, where $\Omega$ is any linear order. The reader who is familiar with the aforementioned would benefit by skipping the proof of the following theorem for it is almost a faithful copy of the original.

\begin{lemma} \label{lem:uncountable} Let $\Omega$ be a linear order and $P \leq \Aut(\Omega)$ be doubly transitive, closed under piecewise patching and disjoint patching. For any interval $\Lambda \subset \Omega$ there exists a representation $\hat{G}_\mathfrak{c}$ of $G_\mathfrak{c}$ in $P$ for which any element in $\hat{G}_\mathfrak{c}$ is the identity outside $\Lambda$.
\end{lemma}

\begin{proof} Begin by taking any interval $\Lambda \subseteq \Omega$ and a copy of $\N$, $\Gamma = \{a_i \in \Lambda \mid i \in \N$ and $a_i < a_{i+1}\}$, in $\Lambda$. By virtue of Lemma~\ref{lem:interval} and for each $i \in \Z$ we are guaranteed a representation of $G_\omega$ within $\Aut_P(a_i,a_{i+1})$ which can be trivially extended to a representation, $G_i$, of $G_\omega$ in $P$. For each $i$, let $f_{i,0}, f_{i,1}, \ldots$ be a set of generators of $G_i$. Next take any $\ad$ family $\mathcal{A} = \{ A_\gamma \subset \N \mid \gamma \in \mathfrak{c}\}$ and for each $\gamma \in \mathfrak{c}$ let $h_\gamma : \N \rightarrow A_\gamma$ be the function that enumerates $A_\gamma$. The crucial step in this proof is to let $g_\alpha$, for any $\alpha \in \mathfrak{c}$, be the permutation on $\Omega$ that when restricted to $(a_i, a_{i+1})$ is exactly $f_{i,h_\alpha(i)}$ and the identity otherwise. This is possible since $P$ is closed under disjoint patching and for $k \not = l$, $\supp(f_{l,n}) \cap  \supp(f_{k,m}) = \emptyset$. To check that indeed $\{g_\alpha \mid \alpha \in \mathfrak{c}\}$ generates $G_\mathfrak{c}$ in $P$ take any word $w(g_{\alpha_0}, \ldots, g_{\alpha_n})$ on distinct $\alpha_0, \ldots, \alpha_n \in \mathfrak{c}$. Recall that by almost disjointness of $\mathcal{A}$ we have that $|A_{\alpha_0} \cap \ldots \cap A_{\alpha_n}| \in \N$ and consequently any pair $g_{\alpha_k}, g_{\alpha_l}$ can agree on at most finitely many intervals. In other words, $g_{\alpha_k} \upharpoonright (a_i, a_{i+1}) =g_{\alpha_l} \upharpoonright (a_i, a_{i+1})$ for at most finitely many $i \in \N$. Clearly, this argument also holds for any finite collection of elements from $\{g_\alpha \mid \alpha \in \mathfrak{c}\}$ and the proof is complete.

\end{proof}

Just as was done with Lemma~\ref{cor:countable} it is possible to extend the previous lemma to a pathological representation of $G_\mathfrak{c}$. Indeed, we must only take a coterminal collection of disjoint intervals from $\Omega$ and notice that for each such interval, say $\Lambda$, and any such $P$, Lemma~\ref{lem:uncountable} can be easily applied to $\Aut_P(\Lambda)$. Consequently, we can build a choice function in very much the same spirit as with Lemma~\ref{cor:countable}.

\begin{corollary} For $\Omega$ a linear order we have that any doubly transitive subgroup of  $\Aut(\Omega)$ closed under piecewise patching and disjoint patching contains a pathological representation of $G_\mathfrak{c}$.

\end{corollary}

What is not obvious is that if the permutation group P is $n$-transitive (for some $n\in \N$) then above representation can be designed to be $n$-transitive provided, of course, that $|\Omega| \leq \mathfrak{c}$. In view of Theorem~\ref{thm:main}, Theorem~\ref{thm:forRandQ} might seem redundant. That said, we think it greatly facilitates the understanding of Theorem~\ref{thm:main}. The technique employed in proving the following is modeled on a result by McCleary (\cite{MR787955} pg. 2).

\begin{theorem} \label{thm:forRandQ} Let $|\Omega| \leq \mathfrak{c}$ and $P \leq \A(\Omega)$ be $n$-transitive, closed under piecewise patching and disjoint patching. Then $G_\lambda$ can be represented as a pathological $n$-transitive permutation group within $P$ provided $\mathfrak{c} \geq \lambda \geq |\Omega|$. In particular, any large subgroup of $\Aut(\Q)$ (resp. $\A(\R)$) contains a pathological $\omega$-transitive representation of $G_\lambda$ (resp. $G_\mathfrak{c}$).

\end{theorem}

\begin{proof} We prove the theorem for $2$-transitivity for it can be easily extended to any $n$-transitivity. Let $P$ and $\lambda$ be given, $\C$ be a coterminal collection of pairwise disjoint bounded intervals from $\Omega$ and denote $\Omega_2 = \{ (a,b) \in \Omega\times \Omega \mid a <b\}$. Let $F: \Omega_2 \times \Omega_2 \rightarrow \lambda$ be any injection. For any pair $(a_1, a_2), (b_1, b_2) \in \Omega_2$ take any bounded interval $\Lambda$ containing the points $a_1, a_2, b_1$ and $b_2$ for which there exists an $f \in P$ so that $f(a_i) = b_i$ and $f(x) = x$ whenever $x \in \Omega \smallsetminus \Lambda$; this is indeed possible since $P$ is closed under piecewise patching. For each interval $I_j \in \C$, let $\{f_{i,j}\}_{i \in \lambda} \subset P$ be the generating set of a representation of $G_\lambda$ for which any generator is the identity outside $I_j$. This is indeed possible by virtue of Lemma~\ref{lem:uncountable}. Next, let $\alpha = F[((a_1,a_2),(b_1,b_2))]$ and $g_\alpha \in P$ be the one for which $g_\alpha \upharpoonright \Lambda = f \upharpoonright \Lambda$, $g_\alpha \upharpoonright I_j = f_{\alpha,j} \upharpoonright I_j$ and the identity otherwise. For all other $\beta \in \Lambda$ not in the range of $F$ we simply let $g_\beta \upharpoonright I_j = f_{\beta, j} \upharpoonright I_j$ and the identity otherwise. Finally, the set $\{g_\alpha \mid \alpha \in \lambda\}$ generates a pathological $2$-transitive representation of $G_\lambda$ within $P$. The extension to $n$-transitivity (resp. $\omega$-transitivity) is done by extending $F$ to an injection from $\displaystyle \bigcup_{j\leq n}\Omega_j^j$ into $\lambda$ (resp. by extending $F$ to an injection from $\displaystyle \bigcup_{n\in\N} \Omega_n^n$ into $\lambda$), where $\Omega_j^j$ represents the j$^{th}$ cartesian product of $\Omega_j=\{(a_1,\ldots,a_j)\in\Omega^j\mid a_i<a_{i+1}\}$.

Recall that any $2$-transitive $l$-group is $\omega$-transitive. Consequently, any large permutation group (on a linear order) is then $\omega$-transitive.

\end{proof}

We are now ready to state and prove the main theorem.

\begin{theorem}\label{thm:main} For $\Omega$ a linear order and $2\leq n\leq\omega$ we have

\begin{enumerate}
\item Any $n$-transitive $P \leq \A(\Omega)$ closed under piecewise patching and disjoint patching contains a pathological representation of $G_{\lambda}$ for $2 \leq \lambda \leq \mathfrak{c}$. Moreover, if $|\Omega| \leq \lambda$ then $G_\lambda$ can be represented as pathological and $n$-transitive within $P$.

\item ($\mathfrak{c}$ regular) If there exists a collection of $\kappa$ disjoint intervals from $\Omega$ then for any $n$-transitive $P \leq \A(\Omega)$ closed under piecewise patching and disjoint patching and any $n < \omega$ for which $\lambda = \mathfrak{c}_n \leq \kappa$ (resp. $\lambda = \mathfrak{c}_\omega \leq \kappa$) there exists a pathological representation of $G_{\lambda^+}$  (resp. $G_\lambda$) in $P$. Moreover, if $\lambda = |\Omega|$ then $G_\lambda$ and $G_{\lambda^+}$ (resp. $G_\lambda$) can be represented as pathological and $n$-transitive within $H$.

\item  (GCH) If for some $n \leq \omega$, $|\Omega| =$ \emph{cof}$(\Omega)$$= \aleph_n$ then any $n$-transitive $P \leq \A(\Omega)$ closed under piecewise patching and disjoint patching contains pathological $n$-transitive representations of $G_{{\aleph_n}}$. Moreover, if $n \in \N$ then the same is true of $G_{2^{\aleph_n}}$.

\end{enumerate}

\end{theorem}

\begin{proof} Theorem~\ref{thm:forRandQ} proves $(1)$. For $(2)$, we first prove the above for all $n\in \N$ by induction on $n$. Let $\mathfrak{c} = \lambda \leq \kappa$ and notice that since $\lambda$ is regular then by Theorem~\ref{thm:kunen} there exists an $\ad$ family $\mathcal{A} = \{A_{\gamma} \mid \gamma \in \lambda^+\} \subset \PP(\lambda)$. Further, let $H \leq \Aut(\Omega)$ be large, $h_{\alpha}: \lambda \rightarrow A_{\alpha}$ be the function that enumerates $A_{\alpha} \in \mathcal{A}$ and $C = \{ I_\alpha \subset \Omega \mid \alpha \in \kappa\}$ be a collection of pairwise disjoint intervals. By virtue of Lemma~\ref{lem:uncountable}, given any interval $I \in C$ there exists $\hat{G}_\mathfrak{c} \leq H$ so that any element in $\hat{G}_\mathfrak{c}$ is the identity outside $I$. Take any $I_\alpha \in C_{\lambda} = \{I_\gamma \in C \mid \gamma \leq \lambda\}$ and let $f_{\alpha 0}, f_{\alpha 1}, \ldots f_{\alpha \omega} \ldots$ be the generators of a representation of $G_\mathfrak{c}$ in $H$, $G(\alpha)$, for which any element in $G(\alpha)$ is the identity outside $I_\alpha$. Given any $\beta \in \lambda^+$ define $g_{\beta}$ to be the permutation on $\Omega$ that when restricted to $I_\alpha$ is exactly $f_{\alpha h_{\beta}(\alpha)}$ and the identity for any $I_\gamma \in C\smallsetminus C_\lambda$. Again, this is possible since intervals in $\mathcal{C}$ are disjoint and $H$ is closed under disjoint patching. To this end we need only check that no word $w(g_{\alpha_1}, \ldots , g_{\alpha_n})$ on distinct $\alpha_1, \ldots , \alpha_n \in \lambda^+$ is trivial. By almost disjointness and regularity of $\lambda$ we know that $|A_{\alpha_1} \cap \ldots \cap A_{\alpha_n}| < \lambda$. Now, for each $\beta \in \lambda^+$ and $I_\alpha \in C_\lambda$ the choice of generator for $g_\beta$ from $F_\alpha$ when restricted to $I_\alpha$ is done by $h_{\beta}$ (i.e. $g_\beta \upharpoonright I_\alpha = f_{\alpha h_\beta(\alpha)}$). In other words, $g_\beta \upharpoonright I_\alpha$ is the $h_{\beta}(\alpha)^{th}$ generator of $F_\alpha$. The aforementioned tells us that for a pair $g_{\alpha_i}, g_{\alpha_j}$ their restrictions to intervals in $C_\lambda$ can be the same on at most fewer than $\lambda$-many intervals. In fact, they can agree on at most as many elements as $A_{\alpha_i}$ and $A_{\alpha_j}$ have in common. Clearly the argument also holds for any finite collection of elements from $\lambda^+$. In turn, we have a $\gamma \in \lambda$ so that $w(g_{\alpha_1}, \ldots , g_{\alpha_n}) \upharpoonright I_\gamma$ is not the identity.

The inductive step is handled in much the same manner. Let $\mathfrak{c} < \lambda = \mathfrak{c}_{n} \leq \kappa$ and notice that since successor cardinals are regular then so must be $\lambda$. Consequently, the logic behind the base case applies to the inductive step.

For the case where $\lambda = \mathfrak{c}_\omega \leq \kappa$ we will construct a nested sequence of free groups within $H$ of increasing rank. Take $\C = \{I_\alpha \mid \alpha \in \mathfrak{c}_\omega\}$ to be a collection of $\mathfrak{c}_\omega$ pairwise disjoint bounded intervals in $\Omega$. For the sake of simplicity, partition $\C$ into $\omega$ many parts of size $\mathfrak{c}_\omega$ so that each part is unbounded in $\Omega$. That is, calling each part $C_n$ ($n \in \omega$), for any $a \in \Omega$ there exist intervals $I_a, I^a \in C_n$ so that $I^a > a$ and $I_a < a$. Notice that by virtue of the previous paragraph, for any $n \in \omega$ we can construct a representation $\hat{G}_{\mathfrak{c}_n} \leq H$ of $G_{\mathfrak{c}_n}$ so that each $g \in \hat{G}_{\mathfrak{c}_n}$ has unbounded support and acts as the identity outside $\bigcup C_n$. For each $n \in \omega$ let $\{f_{n,j}\}_{j \in \mathfrak{c}_n}$ denote a generating set of $\hat{G}_{\mathfrak{c}_n}$. Of course, each $f_{n,j}$ has unbounded support. Next, for all $j \in \mathfrak{c}$ let $g_{0,j} \in H$ be the one for which $g_{0,j}(x) = f_{n,j}(x)$ when $ x \in \bigcup C_n$ and the identity otherwise. In general, we want for all $k \in \omega$

\begin{equation*}
g_{k,j}(x) =
\begin{cases}

f_{n,j}(x) & \text{if } x \in \bigcup C_n \mbox{ and } j \in \mathfrak{c}_n \\
x & \text{otherwise.}\\
\end{cases}
\end{equation*}\\

Consequently, for each $k \in \omega$, $\{g_{k, j}\}_{j \in \mathfrak{c}_k}$ generates a pathological copy of $G_{\mathfrak{c}_k}$ in $H$ and for $m<n$ we have that $\{g_{m, j}\}_{j \in \mathfrak{c}_m} \subset \{g_{n, j}\}_{j \in \mathfrak{c}_n}$. To this end it should be clear that $\bigcup\limits_{n \in \omega} \{g_{n, j}\}_{j \in \mathfrak{c}_n}$ generates a pathological copy of $G_{\mathfrak{c}_\omega}$ in $H$.

Next, if $|\Omega| = \lambda$ it is a simple matter to set up a choice function between generators of $G_\lambda$ and ordered pairs within $\Omega$ in much the same fashion as with Theorem~\ref{thm:forRandQ}.

Lastly, for $(3)$ since GCH implies that $\mathfrak{c}$ is regular then $(2) \Rightarrow (3)$.

\end{proof}

\vspace{.1 in}

 For some $n \in \N$, let $L(\aleph_n) = \aleph_n \times \aleph_n$ ordered lexicographically. Assuming GCH and by Theorem~\ref{thm:main} there exists a representation of the free group of rank $\aleph_n^+ = 2^{\aleph_n} = |\Aut(L(\aleph_n))|$ within any large subgroup of $\Aut(L(\aleph_n))$. Thus, under GCH, $|H| = |\Aut(L(\aleph_n))|$ for any large $H$.

Nested sequences of free groups are essential for the construction of $G_{\aleph_\omega}$ in Theorem~\ref{thm:main}. Take for instance the following theorem where $\bigoplus$ denotes a restricted direct product.

\begin{theorem} \label{thm:freeprods} Let $G = \bigoplus_{\beta \in \kappa} H_{\beta}$ for which $|H_{\alpha}| < sup(|H_{\beta}|)_{\beta \in \kappa} = \lambda$, for any $\alpha \in \kappa$. Then for any freely generated $H \leq G$ we have that rank$(H) < \lambda$.

\end{theorem}

\begin{proof} Assume that there exists $H \leq G$ freely generated by $F_H$ with $|F_H| = \lambda$. In order to avoid trivialities we assume $\lambda > \aleph_0$. Let us begin by defining $i : F_H \rightarrow [\kappa]^{< \omega}$ where $a \in i (g)$ iff $\pi_{a}(g) \not = e$ (i.e. the $a^{th}$ coordinate of $g$ is not the identity). Take any $A \in ran(i)$  and notice that for any $B \in ran(i)$, $A \cap B \not = \emptyset$. Otherwise all elements from $i^{-1}(A)$ commute with any element from $i^{-1}(B)$. Define, for all $a \in \kappa$ and $h \in H_a$, $a(h) = \{ g \in F_H \mid \pi_a(g) = h\}$ and given $B \in ran(i)$ let $B_a(h) = i^{-1}(B) \cap a(h)$. Back to $A$, we have that since $A \in [\kappa]^{< \omega}$ then by the Pigeonhole Principle there exists $a_1 \in A$ and $h_1 \in H_{a_1}$ for which

$$ \left| \bigcup_{B \in ran(i)} B_{a_1}(h_1) \right| = \lambda$$

where we let $a_1$ be the smallest element of $A$ for which the above is true. Next, denote $\B_1 = \bigcup\limits_{B \in ran(i)} B_{a_1}(h_1)$ and notice that there must exist $\lambda$  elements $f \in \B_1$ so that $i(f) \not \subset A$ (since the largest element in $\{|H_a|\}_{a \in A}$ is smaller that $\lambda$). Take any pair $g_1, g_2 \in \B_1$ and observe that if $i(g_1) \cap A = i(g_2) \cap A = \{a_1\}$ then $[g_1,g_2] := g_1g_2g_1^{-1}g_2^{-1}$ commutes with all elements from $i^{-1}(A)$. Thus, for all (but at most one) $g \in \B_1$, $i(g) \cap A - \{a_1\} \not = \emptyset$. Let $A_1 = A - \{a_1\}$. Since $A_1 \in [\kappa]^{< \omega}$ then we can find an $a_2 \in A_1$ and $h_2 \in H_{a_2}$ so that
$$ \left| \bigcup_{B \in ran(i)} (B_{a_1}(h_1) \cap B_{a_2}(h_2))\right| = \lambda.$$

Again, we let $a_2$ be the smallest element in $A_1$ that satisfies the above equation. We are now in a very similar situation to that encountered at the beginning of the proof. In turn, we can run the above argument until we exhaust all of $A$, at which point we have a collection of elements from $H$ whose commutator commute with any element from $A$. Indeed, let us assume that $A=\{a_1, \ldots, a_k\}$ and that
$$ \left| \bigcup_{B \in ran(i)}\left( \bigcap_{j\leq k-1} B_{a_j}(h_j)\right) \right| = \lambda \mbox{ for } h_j \in H_{a_j}.$$

 Let $\B_k =  \bigcup\limits_{B \in ran(i)} \left( \bigcap\limits_{j\leq k} B_{a_j}(h_j)\right)$ and again take any pair $g_1, g_2 \in \B_k$ for which $i(g_1)$ and $i(g_2)$ are not contained in $A$. If $a_k \not \in i(g_1),i(g_2)$ then $[g_1,g_2]$ commutes with everything in $A$. In turn, we must have that for all (but at most one) $g \in \B_k$, $a_k \in i(g)$. Hence, for the last element in $A$ we have
$$ \left| \bigcup_{B \in ran(i)} \left(\bigcap_{j\leq k} B_{a_j}(h_j) \right)\right| = \lambda \mbox{ for }  h_{j} \in H_{a_{j}}.$$

Lastly, take any pair $g_1,g_2 \in \bigcup\limits_{B \in ran(i)} \left(\bigcap_{j\leq k} B_{a_j}(h_j)\right)$ for which $i(g_1),i(g_2) \not \subseteq A$ and notice that $[f,[g_1,g_2]] = e$ for any $f \in i^{-1}(A)$.

\end{proof}

\begin{corollary} \label{thm:res1} Let $G = \bigoplus_{i \in \omega} F_i$ so that each $F_i$ is the free group of rank $ \aleph_i$. If $H \leq G$ is a free group then $rank( H) < \aleph_{\omega}$.

\end{corollary}

\section{Conclusion}
The set-theoretic restriction on Theorem~\ref{thm:main} (i.e. singular cardinals) highlights a natural bound in the method employed within this paper and it is unknown to the present author whether or not it can be extended any further. As far as transitivity goes, the minimum requirement we must impose on a permutation group is double transitivity on at least one interval of the linear order. Also, piecewise patching and disjoint patching are very strong
properties. An interesting result would involve free group constructions involving weaker versions of the above properties.
As for free group representations, a similar question can be asked in terms of $l$-groups. That is, when do permutation groups of linear orders admit representations of free $l$-groups? It is worth noting that the Ping-Pong Lemma does not guarantee a lattice order on the free group.

\newpage
\bibliographystyle{amsplain}
\bibliography{mybib}

\providecommand{\bysame}{\leavevmode\hbox to3em{\hrulefill}\thinspace}
\providecommand{\MR}{\relax\ifhmode\unskip\space\fi MR }
\providecommand{\MRhref}[2]{%
  \href{http://www.ams.org/mathscinet-getitem?mr=#1}{#2}
}
\providecommand{\href}[2]{#2}
\begin{thebibliography}{10}

\bibitem{MR1363412}
Curtis~D. Bennett, \emph{Explicit free subgroups of {${\rm Aut}({\bf
  R},\leq)$}}, Proc. Amer. Math. Soc. \textbf{125} (1997), no.~5, 1305--1308.
  \MR{1363412 (97g:06020)}

\bibitem{MR2796250}
V.~V. Bludov, M.~Droste, and A.~M.~W. Glass, \emph{Automorphism groups of
  totally ordered sets: a retrospective survey}, Math. Slovaca \textbf{61}
  (2011), no.~3, 373--388. \MR{2796250}

\bibitem{MR2796249}
V.~V. Bludov and A.~M.~W. Glass, \emph{Right orders and amalgamation for
  lattice-ordered groups}, Math. Slovaca \textbf{61} (2011), no.~3, 355--372.
  \MR{2796249}

\bibitem{Cameron99oligomorphicpermutation}
Peter~J. Cameron, \emph{Oligomorphic permutation groups}, London Mathematical
  Society Student Texts \textbf{45} (1999).

\bibitem{MR0091280}
P.~M. Cohn, \emph{Groups of order automorphisms of ordered sets}, Mathematika
  \textbf{4} (1957), 41--50. \MR{0091280 (19,940e)}

\bibitem{MR0270992}
Paul Conrad, \emph{Free lattice-ordered groups}, J. Algebra \textbf{16} (1970),
  191--203. \MR{0270992 (42 \#5875)}

\bibitem{MR988099}
M.~Droste, W.~C. Holland, and H.~D. Macpherson, \emph{Automorphism groups of
  infinite semilinear orders. {I}, {II}}, Proc. London Math. Soc. (3)
  \textbf{58} (1989), no.~3, 454--478, 479--494. \MR{988099 (90b:20006)}

\bibitem{MR0409309}
A.~M.~W. Glass, \emph{{$l$}-simple lattice-ordered groups}, Proc. Edinburgh
  Math. Soc. (2) \textbf{19} (1974/75), no.~2, 133--138. \MR{0409309 (53
  \#13069)}

\bibitem{MR0158009}
Charles Holland, \emph{The lattice-ordered groups of automorphisms of an
  ordered set}, Michigan Math. J. \textbf{10} (1963), 399--408. \MR{0158009 (28
  \#1237)}

\bibitem{MR756630}
Kenneth Kunen, \emph{Set theory}, Studies in Logic and the Foundations of
  Mathematics, vol. 102, North-Holland Publishing Co., Amsterdam, 1983, An
  introduction to independence proofs, Reprint of the 1980 original. \MR{756630
  (85e:03003)}

\bibitem{MR787955}
Stephen~H. McCleary, \emph{Free lattice-ordered groups represented as
  {$o$}-{$2$} transitive {$l$}-permutation groups}, Trans. Amer. Math. Soc.
  \textbf{290} (1985), no.~1, 69--79. \MR{787955 (86m:06034a)}

\bibitem{pingpong}
Uri Weiss, \emph{Ping-pong lemma (version 5)}, PlanetMath.org.

\bibitem{MR969681}
Samuel White, \emph{The group generated by {$x\mapsto x+1$} and {$x\mapsto
  x^p$} is free}, J. Algebra \textbf{118} (1988), no.~2, 408--422. \MR{969681
  (90a:12014)}

\end{thebibliography}
\end{document}